\documentclass[11pt]{amsart}

\usepackage{amsmath}
\usepackage{amscd}
\usepackage{amssymb}
\usepackage{amsthm}
\usepackage{latexsym}
\usepackage{mathabx}
\usepackage{mathtools}
\usepackage{hyperref}
\usepackage{stmaryrd}
\usepackage{tikz-cd}
\usepackage{enumitem}
\usepackage{accents}
\usepackage{dsfont}

\usepackage[arrow,curve,matrix,tips,2cell]{xy}
  \SelectTips{eu}{10} \UseTips
  \UseAllTwocells

\usepackage{calrsfs}
\usepackage[T1]{fontenc} 
\usepackage{textcomp}
\usepackage{times}
 \usepackage[scaled=0.92]{helvet}

\usepackage{geometry}
\newtheorem{theorem}{Theorem}[section]
\newtheorem{lemma}[theorem]{Lemma}
\newtheorem{corollary}[theorem]{Corollary}

\newtheorem{theoremA}{Theorem}

\theoremstyle{definition}
\newtheorem{remark}[theorem]{Remark}
\newtheorem{definition}[theorem]{Definition}
\theoremstyle{plain}

\newtheorem*{cor}{Corollary}

\newcommand{\cP}{\mathcal P}

\newcommand{\cS}{\mathcal S}

\def\Cz{\mathbb{C}}

\def\Nz{\mathbb{N}}
\def\Qz{\mathbb{Q}}

\def\Zz{\mathbb{Z}}

\newcommand{\mfp}{\mathfrak p}
\newcommand{\mfq}{\mathfrak q}
\newcommand{\mfP}{\mathfrak P}

\newcommand{\ab}{\textrm{ab}}

\newcommand{\Gal}{\mathrm{Gal}}

\newcommand{\Gkab}{G^{\ab}_K}
\newcommand{\Gkabx}{G^{\ab}_{K'}}
\newcommand{\Gnab}{G^{\ab}_N}
\newcommand{\Hkab}{\widecheck{G}_{K}}
\newcommand{\Hlab}{\widecheck{G}_{K}}

\newcommand{\Hlabx}{\widecheck{G}_{K'}}
\newcommand{\Hnab}{\widecheck{G}_{N}}
\newcommand{\Frob}[1]{\textup{Frob}_{#1}}

\newcommand{\IsoL}{\textup{Iso}_{L\textup{-series}}(\Hlab[l], \Hlabx[l])}
\newcommand{\IsoP}{\textup{Iso}_{\mathcal{P}}(\Hlab[l], \Hlabx[l])}
\newcommand{\IsoPx}{\textup{Iso}^\delta_{\mathcal{P}}(\Hlab[l], \Hlabx[l])}

\newcommand{\Isofield}{\textup{Iso}(K, K')}

\newcommand{\DistTo}{\xrightarrow{
   \,\smash{\raisebox{-0.45ex}{\ensuremath{\scriptstyle\sim}}}\,}}

\newcommand\restr[2]{{
  \left.\kern-\nulldelimiterspace 
  #1 
  \vphantom{\big|} 
  \right|_{#2} 
  }}

\newcommand{\new}[1]{\textcolor{black}{#1}}

\begin{document}

\date{\today\ (version 1.0)} 
\title{$L$-series and isomorphisms of number fields}
\author[H.~Smit]{Harry Smit}
\address{(HS) Mathematisch Instituut, Universiteit Utrecht, Postbus 80.010, 3508 TA Utrecht, Nederland}
\email{h.j.smit@uu.nl}

\begin{abstract}
Two number fields with equal Dedekind zeta function are not necessarily isomorphic. However, if the number fields have equal sets of Dirichlet $L$-series then they \emph{are} isomorphic. We extend this result by showing that the isomorphisms between the number fields are in bijection with $L$-series preserving isomorphisms between the character groups.
\end{abstract}

\maketitle

\section{Introduction}
Kronecker \cite{kronecker} started a programme to characterise a number field and its extensions using only information about the prime ideals (``primes'' from now on). Ga{\ss}mann (\cite[p.\ 671--672]{gassmann}) showed that number fields with the same Dedekind zeta function (called arithmetically equivalent fields) need not be isomorphic. Even all local information does not suffice: Komatsu \cite{komatsu} gave an example of two non-isomorphic number fields with isomorphic adele rings.

Fortunately, there has also been success: the Neukirch-Uchida Theorem (\cite[Satz 2]{neukirch} and \cite[Ch.\ XII, \S 2]{neukirchcohom})  states that two number fields with isomorphic absolute Galois groups are necessarily isomorphic. Uchida (\cite[Main Thm.]{uchida2}, see also \cite[Ch.\ XII, \S 2, Cor.\ 12.2.2]{neukirchcohom}) proved later that the link between a number field and its absolute Galois group is even stronger: the automorphisms of a number field are in bijection with the outer automorphisms of the absolute Galois group. As a drawback, the structure of the absolute Galois group is, even for $\Qz$, not very well understood. 

A reasonable object to consider next is the abelianized absolute Galois group. Although it is well-understood by class field theory, it lacks the capacity to uniquely determine the underlying field: Kubota (\cite[\S 4]{kubota}) gave a classification of the abelianized absolute Galois groups of number fields, which was used by Onabe \cite{onabe} to show that there exist non-isomorphic imaginary quadratic fields with isomorphic abelianized absolute Galois groups. Stevenhagen and Angelakis (\cite[Thm.\ 4.1]{angelakisstevenhagen}) gave an explicit form of the abelianized absolute Galois group for many imaginary quadratic fields of low class number and found that most of these groups were isomorphic. 

For any number field, there exists a Dirichlet $L$-series of a well-chosen character of odd prime order that does not occur as a Dirichlet $L$-series of any other non-isomorphic number field, see \cite[Thm.\ 10.1]{CdSLMS}, and the same holds for a pair of quadratic characters (\cite[Thm.\ 2.2.2]{Pintonello}). Therefore, if two number fields share all $L$-series for characters of a certain prime order, then they are isomorphic. However, these theorems do not provide explicit isomorphisms, hence the question arises whether or not one can link the automorphism group of a number field to its Dirichlet $L$-series.

The main result of this paper is that when one considers the structure of the Dirichlet character group of the absolute Galois group along with their Dirichlet $L$-series, then one can not only recover the underlying number field, but also its automorphism group.

Let $K$ and $K'$ be number fields, and denote by $\Hlab[l]$ and $\Hlabx[l]$ the $l$-torsion of the character groups of their absolute Galois groups. Moreover, let $\IsoL$ be the set of isomorphisms $\psi: \Hlab[l] \DistTo \Hlabx[l]$ such that $\chi$ and $\psi(\chi)$ have the same $L$-series for any $\chi \in \Hlab[l]$. We prove the following theorem:

\begin{theoremA}\label{theoremA:bijection between IsoL and Isofield}
Let $K$ and $K'$ be number fields, and let $l$ be any prime number. There exists a bijection
\begin{equation*}
\begin{tikzcd}
\Theta: \IsoL \arrow{r} & \Isofield.
\end{tikzcd}
\end{equation*}
\end{theoremA}
The case $l=2$ was first proven by Gabriele Dalla Torre in his unpublished PhD thesis. The idea of the construction of the map $\Theta$ is similar to the approach taken by Neukirch and Uchida: they construct a bijection of primes that preserves the decomposition groups inside the absolute Galois groups. Given a map $\psi \in \IsoL$, we first derive a bijection of primes $\phi$ that is compatible with $\psi$. 
\new{What follows is an application of the following theorem:}
\begin{theoremA}\label{theoremA:pos_density}
\new{Let $K$ and $K'$ be number fields of the same degree, and let $\cS$ be a subset of the primes of inertia degree $1$ of $K$. Suppose that for some finite extension $L/K$, $\cS$ contains $\emph{\text{Spl}}(L/K)$ except for a Dirichlet density zero set. Furthermore, suppose there exists an isomorphism $\psi: \Hlab[l] \DistTo \Hlabx[l]$ with an injective norm-preserving map $\phi: \cS \to \cP_{K'}$ such that
\[
\chi(\mfp) = \psi(\chi)(\phi(\mfp))
\]
Then $K \simeq K'$ and there is a unique $\sigma_\psi: K \DistTo K'$ such that the bijection of primes induced by $\sigma_\psi$ equals $\phi$ on $\cS$ except for finitely many exceptions. }
\end{theoremA}

\new{To complete the proof of Theorem~\ref{theoremA:bijection between IsoL and Isofield}, we check that the map that sends $\psi$ to $\sigma_\psi$ is a bijection.}

The proofs of both theorems rely heavily on the Grunwald-Wang theorem (see \cite[Ch.\ X, Thm.\ 5]{artintate}), which allows for the creation of characters with specific values on any finite set of primes. Furthermore, we use the Chebotarev density theorem (\cite[Ch.\ VII, \S 13, p.\ 545]{neukirch2013algebraic}) to bound degrees of extensions.

Lastly, we state a corollary that can be seen as an analogue of theorems about different types of equivalence of number fields (such as arithmetical or Kronecker equivalence, \cite[Ch.\ II \& III]{klingenbook}), see for example the Main Theorem of \cite{perlisstuart} and \cite[Satz 1]{klingen}. Both theorems guarantee that no density zero set of ``exceptional primes'' can exist.

\begin{cor}
Let $K$ and $K'$ be number fields, and let $\cS$ be a set of primes of $K$ of Dirichlet density one. Suppose there is an isomorphism $\psi: \Hlab[l] \DistTo \Hlabx[l]$ with an injective norm-preserving map $\phi: \cS \to \cP_{K'}$ such that
\[
\chi(\mfp) = \psi(\chi)(\phi(\mfp))
\]
for any $\chi \in \Hlab[l]$ and $\mfp \in \cS$. Then $\phi$ can be uniquely extended to a norm-preserving bijection between all primes such that $\chi(\mfp) = \psi(\chi)(\phi(\mfp))$ for any $\chi \in \Hlab[l]$ and any prime $\mfp$ of $K$.
\end{cor}

A natural follow-up question that is not considered in this paper is whether or not Theorem~\ref{theoremA:bijection between IsoL and Isofield} can be strengthened to include not just isomorphisms, but homomorphisms between number fields.

\section{Preliminaries}
In this section we set up notation, introduce the main objects of study, and state some convenient lemmas.

We fix an algebraic closure $\overline{\Qz}$ of $\Qz$ throughout the entire paper. We denote number fields by $k$, $K$, $K'$, and $N$, where usually $N/\Qz$ is a Galois extension. We use $\mfp$ for the prime ideals of $K$ (that we will call ``primes''), $\mfq$ for the primes of $K'$, and $\mfP$ for the primes of $N$. 

The set of primes of a number field $K$ is denoted $\cP_{K}$, and the set of primes lying over a rational prime $p$ is denoted $\cP_{K, p}$. Given a prime $\mfp \in \cP_K$, we denote the norm of $\mfp$ by $N\mfp := \#\big(\mathcal{O}_K/\mfp\big) = p^{f_{\mfp}}$, where $f_{\mfp}$ is the inertia degree of $\mfp$. We denote the zeta function of $K$ by $\zeta_K$.

For a field $K$, let $G_K$ be the absolute Galois group, $K^\textup{ab}$ be the composite of all abelian extensions of $K$, and denote by $\Gkab$ its Galois group over $K$. The dual, denoted $\Hlab$, is the group of all Dirichlet characters (i.e. continuous homomorphisms) $G_K \to \Cz^\times$. For any prime number $l$ we denote by $\Hlab[l]$ the subgroup of $\Hlab$ generated by characters of order $l$. Denote by $1_K \in \Hlab$ the trivial character.

We use the following convenient notation: if $P(T)$ is a polynomial, denote by $[T^n]P(T)$ the coefficient of $T^n$ of $P(T)$.

\subsection*{Dirichlet characters} \leavevmode \\
Associated to every Dirichlet character $\chi \in \Hkab$ is a unique finite cyclic extension $K_{\chi}/K$ of degree equal to the order of $\chi$ such that $\chi$ factors through $\Gal(K_{\chi}/K)$. Let $\mfp$ be a prime of $K$ unramified in $K_{\chi}/K$ and let $\Frob{\mfp}$ be the Frobenius element in $\Gal(K_{\chi}/K)$. We set $\chi(\mfp) := \chi(\Frob{\mfp})$. If $\mfp$ is a prime of $K$ that ramifies in $K_{\chi}/K$, we set $\chi(\mfp) = 0$.

Throughout the paper we will be concerned with the existence of characters with certain properties (mainly with prescribed values at specific primes). For our purposes, the question of whether or not such characters exist is answered by the Grunwald-Wang theorem \cite[Ch.\ X, Thm.\ 5]{artintate}. It states the following: let $\mfp_1, \dots, \mfp_n$ be primes of $K$, and let $K_{\mfp_i}$ be the localization of $K$ at $\mfp_i$. For any $1 \leq i \leq n$, let $m_i$ be an integer and let $\chi_i \in \widecheck{G}_{K_{\mfp_i}}[m_i]$. Then, aside from a special case that occurs only when $m := \text{lcm}(m_1, \dots, m_n)$ is divisible by $4$, there exists a character $\chi \in \Hkab[m]$ such that $\chi(\mfp_i) = \chi_i(\mfp_i)$. 

We do not use the Grunwald-Wang theorem in its full generality; the following lemma suffices. 

\begin{lemma}\label{lemma:grunwald-wang in any order}
Let $K$ be a number field and $S = \{\mfp_1, \dots, \mfp_s\}$ a finite set of primes of $K$. Let $l$ be a prime number, and let $\zeta_1, \zeta_2, \dots, \zeta_s$ be $l^\text{th}$ roots of unity. Then there exists a character $\chi \in \Hlab[l]$ such that $\chi(\mfp_i) = \zeta_i$ for all $1 \leq i \leq s$. 
\end{lemma}

\begin{proof}
Any of the $K_{\mfp_i}$ has an unramified Galois extension $N$ of degree $l$ obtained by adjoining a $(N\mfp_i^l - 1)^\text{th}$ root of unity. Therefore there exists a local character $\chi_i$ that factorises through $\Gal(N/K_{\mfp_i})$ such that $\chi_i(\mfp_i) = \zeta_i$. The existence of $\chi$, along with its order, is now guaranteed by the Grunwald-Wang theorem.
\end{proof}

\subsection*{The $L$-series of a Dirichlet character}

\begin{definition}
Associated to any Dirichlet character $\chi \in \Hkab$ is an $L$-series $L(\chi, s)$, defined by
\[
L(\chi, s) = \prod_{\mfp \in \cP_K} \big(1 - \chi(\mfp) p^{- f_{\mfp} s}\big)^{-1}.
\]
\end{definition}
We define $L_\mfp(\chi, T) = 1 - \chi(\mfp) T^{f_{\mfp}}$, so that
\[
L(\chi, s) = \prod_{p \in \cP_{\Qz}} \prod_{\mfp \in \cP_{K,p}} L_\mfp(\chi, p^{-s})^{-1}.
\]
The $L$-series can also be written as an infinite series:
\begin{equation} \label{eq:L-series}
L(\chi, s) = \sum_{n = 1}^\infty \frac{a_n}{n^s},
\end{equation}
with $a_n \in \Cz$. Any $L$-series converges for any $s$ with $\text{Re}(s) > 1$, and two $L$-series are equal if and only if all their coefficients are the same (see \cite[Ch.\ 2, \S 2.2, Cor.\ 4]{serrecourse}), i.e. 
\[
L(\chi, s) = L(\chi', s) \Longleftrightarrow [n^{-s}]L(\chi, s) = [n^{-s}]L(\chi', s) \text{ for all $n \in \mathbb{N}$,}
\]
where $[n^{-s}]L(\chi, s)$ denotes $a_n$ in the representation (\ref{eq:L-series}). Lastly, we remark that the $L$-series of the trivial character is the Dedekind zeta function of the number field.

\section{Main theorem}
The main theorem establishes bijections between the following four sets of isomorphisms:
\begin{definition} 
\mbox{ }
\begin{itemize}
\item $\Isofield$ is the set of field isomorphisms $K \DistTo K'$. 
\item $\textup{Iso}_{L\textup{-series}}(\Hlab[l], \Hlabx[l])$ is the set of isomorphisms $\psi: \Hlab[l] \DistTo \Hlabx[l]$ for which 
\[
L(\chi, s) = L(\psi(\chi), s)
\]
holds for all characters $\chi \in \Hlab[l]$. 
\item $\IsoP$ is the set of isomorphisms $\psi: \Hlab[l] \DistTo \Hlabx[l]$ for which there is a norm-preserving bijection $\phi: \mathcal{P}_K \to \mathcal{P}_{K'}$ such that 
\[
\chi(\mfp) = \psi(\chi)(\phi(\mfp))
\]
for all $\chi \in \Hlab[l]$ and $\mfp \in \cP_K$. 
\item $\IsoPx$ is the set of isomorphisms $\psi: \Hlab[l] \DistTo \Hlabx[l]$ such that there is a set $\cS \subseteq \cP_K$ of primes of Dirichlet density one and an injective norm-preserving map $\phi: \cS \to \cP_{K'}$ such that for any $\mfp \in \cS$ and any $\chi \in \Hlab[l]$ the equality
\[
\chi(\mfp) = \psi(\chi)(\phi(\mfp))
\]
holds.
\end{itemize}
\end{definition}

\begin{theorem}\label{theorem:main theorem}
\new{Let $l$ be any prime number. There exist injective maps}
\begin{equation*}
\begin{tikzcd}
\IsoL \arrow[r, shift left, "\Theta_3"]  &  \arrow{d}{\Theta_1} \arrow[l, shift left, "\Theta_2"] \IsoP & \arrow[l, swap, "\Theta_4"] \Isofield \\
 & \arrow[ur, swap, "\Theta_5"] \IsoPx & 
\end{tikzcd}
\end{equation*}
such that $\Theta_2$ and $\Theta_3$ are mutual inverses, and $\Theta_1 \circ \Theta_4$ and $\Theta_5$ are mutual inverses. As a result, the sets of isomorphisms are all in bijection.
\end{theorem}

The remainder of the article is structured as follows. Section~\ref{section:structure} can be seen as a ``nuts and bolts'' section on isomorphisms between character groups. Maps $\Theta_1$, $\Theta_2$, and $\Theta_3$ are all the identity map, and their properties are proven in Section~\ref{section:maps 1 to 3}, while Section~\ref{sec:maps 4 and 5} deals with maps $\Theta_4$ and $\Theta_5$. Map $\Theta_4$ does not require difficult techniques as any isomorphism of fields $K \DistTo K'$ induces an isomorphism $\Hlab \DistTo \Hlabx$ along with a bijection of primes. Map $\Theta_5$ requires significantly more work: the main idea is to first find an automorphism of the Galois closure of $K$ with properties concerning the bijection of primes attached to any element of $\IsoP$. We then show that this automorphism restricts to an isomorphism $K \DistTo K'$. 

\section{Isomorphisms between character groups}\label{section:structure}
The aim of this section is to show that any isomorphism $\psi \in \IsoP$ is characterised by its associated bijection of primes.

\begin{lemma}\label{lemma:prime-bijection-is-unique}
For any $\psi \in \IsoP$ there is a unique bijection of primes $\phi: \cP_K \to \cP_{K'}$ such that 
\[
\psi(\chi)(\phi(\mfp)) = \chi(\mfp). 
\]
\end{lemma}

\begin{proof}
Indeed, let $\phi_1$ and $\phi_2$ be bijections $\cP_K \to \cP_{K'}$ such that
\[
\psi(\chi)(\phi_1(\mfp)) = \chi(\mfp) = \psi(\chi)(\phi_2(\mfp))
\]
for any $\mfp \in \cP_K$ and $\chi \in \Hlab[l]$, and suppose for a certain prime $\tilde{\mfp} \in \cP_{K}$ that $\phi_1(\tilde{\mfp}) \neq \phi_2(\tilde{\mfp})$. By Lemma~\ref{lemma:grunwald-wang in any order} there exists a $\chi' \in \Hlabx[l]$ such that $\chi'(\phi_1(\tilde{\mfp})) = 1$ and $\chi'(\phi_2(\tilde{\mfp})) \neq 1$. Then
\[
\chi'(\phi_1(\mfp)) = \psi^{-1}(\chi')(\mfp) = \chi'(\phi_2(\mfp))
\]
is an immediate contradiction.
\end{proof}

\begin{lemma}\label{lem:psi-id}
Let $\psi \in \textup{Aut}(\Hlab[l])$ be an automorphism such that $\psi(\chi)(\mfp) = \chi(\mfp)$ for all $\chi \in \Hlab[l]$ and all $\mfp$ in a density one subset of the primes of $K$. Then $\psi$ is the identity. 
\end{lemma}

\begin{proof}
Consider the character $\tilde{\chi} = \psi(\chi)/\chi$. By assumption, $\tilde{\chi}(\mfp) = 1$ for a density one set of primes, hence a density one set of primes splits completely in the extension $K_{\tilde{\chi}}/K$. By the Chebotarev density theorem $[K_{\tilde{\chi}}:K] = 1$, implying $\psi(\chi) = \chi$.
\end{proof}

\begin{corollary}\label{cor:psi1-psi2}
\new{Let $\psi_1, \psi_2 \in \IsoPx$. Suppose their corresponding maps of primes $\phi_1$ and $\phi_2$ agree on a subset of the primes of density $1$. Then $\psi_1 = \psi_2$. }
\end{corollary}

\begin{proof}
Apply the previous lemma to $\psi_1 \psi_2^{-1}$.
\end{proof}

\section{Maps $\Theta_1$, $\Theta_2$, and $\Theta_3$}\label{section:maps 1 to 3}
We prove that the three maps $\Theta_1$, $\Theta_2$, and $\Theta_3$ of Theorem~\ref{theorem:main theorem} can all be chosen as $\psi \mapsto \psi$. This establishes a bijection between the three sets. Each of the following subsections deals with one of the maps.

\subsection*{The map $\Theta_1$} \leavevmode \\
This is a triviality: the conditions imposed on elements of $\IsoP$ are stronger than those imposed on elements of $\IsoPx$. Hence the map $\Theta_1: \psi \mapsto \psi$ is an injective map 
$\IsoP \to \IsoPx$.

\subsection*{The map $\Theta_2$} \leavevmode \\
\new{This can be found in \cite[Section 5]{CdSLMS}. For the sake of self-containedness, we include the argument here as well. Let $\psi \in \IsoP$ with associated norm-preserving bijection of primes $\phi$. Note that for any $\chi \in \Hkab[l]$ and $\mfp \in \cP_K$ we have $\chi(\mfp) = \psi(\chi)(\phi(\mfp))$. Hence }
\[
L_{\mfp}(\chi, T) = 1 - \chi(\mfp) T^{f_{\mfp}} = 1 - \psi(\chi)(\phi(\mfp)) T^{f_{\phi(\mfp)}} = L_{\phi(\mfp)}(\psi(\chi), T).
\]
\new{It follows that $L(\chi, s) = L(\psi(\chi), s)$, hence $\psi \in \IsoL$. Therefore the map $\Theta_2: \psi \mapsto \psi$ is an injective map $\IsoP \to \IsoL$.}

\subsection*{The map $\Theta_3$} \leavevmode \\
We show that to any element $\psi \in \IsoL$ we can attach a norm-preserving bijection of primes $\phi: \cP_{K} \to \cP_{K'}$, from which it follows that the map $\Theta_3$ can be defined as $\psi \mapsto \psi$. For the remainder of this section, suppose there exists an isomorphism $\psi: \Hlab[l] \DistTo \Hlabx[l]$ such that $L(\chi, s) = L(\psi(\chi), s)$ for all $\chi \in \Hlab[l]$. The following definition and lemma allow us, in order to construct $\phi$, to focus only on primes that lie over the same rational prime.

\begin{definition}\label{definition:local factor at p}
Let $\chi \in \Hlab$ and let $p$ a prime. The \emph{local factor at $p$} is
\[
L_p(\chi, T) := \prod_{\mfp \in \cP_{K, p}} L_{\mfp}(\chi, T).
\] 
\end{definition}

\begin{lemma}\label{lemma:Lp equal}
Let $\chi \in \Hlab, \chi'\in \Hlabx$. The equality $L(\chi, s) = L(\chi', s)$ holds if and only if 
\[
L_p(\chi, T) = L_p(\chi', T)
\]
for all rational primes $p$.
\end{lemma}

\begin{proof}
The ``if'' part is clear as $L(\chi, s) = \prod\limits_{p \text{ prime}} L_p(\chi, p^{-s})^{-1}$.

\noindent As $L_p(\chi, p^{-s})$ is a polynomial in $p^{-s}$ and all $L_p(\chi, p^{-s})$ have constant coefficient equal to $1$, we have
\[
[p^{-js}] L(\chi, s) = [p^{-js}] L_p(\chi, p^{-s})^{-1}
\]
for any rational prime $p$ and any $j \in \Nz$.
\end{proof}

As $\psi(1_{K}) = 1_{K'}$ and $\zeta_K(s) = L(1_K, s)$, we find $\zeta_K = L(1_K, s) = L(1_{K'}, s) = \zeta_{K'}$. Hence $K$ and $K'$ have the same number of primes lying over every rational prime. Let $p$ be a rational prime and let $\mfp_1, \dots, \mfp_n$ be the primes of $K$ lying over $p$, and $\mfq_1, \dots, \mfq_n$ the primes of $K'$ lying over $p$.

Lemma~\ref{lemma:Lp equal} asserts that $L_p(\chi, T) = L_p(\psi(\chi), T)$ for any $\chi \in \Hkab[l]$,
which reads
\[
\prod_{i = 1}^n (1 - \chi(\mfp_i)T^{f_i}) = \prod_{j = 1}^n (1 - \psi(\chi)(\mfq_j)T^{f'_j}),
\]
where $f_i$ is the inertia degree of $\mfp_i$ and $f'_j$ is the inertia degree of $\mfq_j$.

Note that the order of zero at $T = 1$ on the left hand side is equal to the number of primes $\mfp$ over $p$ at which $\chi(\mfp) = 1$, while on the right hand side it equals the number of primes $\mfq$ lying over $p$ for which $\psi(\chi)(\mfq) = 1$. This equality proves the following lemma.

\begin{lemma}\label{lemma:psichi only one prime}
Let $\chi_i \in \Hkab[l]$ be any character that has value $\zeta_l = \exp(2\pi i /l)$ on $\mfp_i$ and value $1$ on all other primes lying over $p$. The character $\psi(\chi_i)$ has value $1$ on all primes of $K'$ lying over $p$ except at a single prime.  \qed
\end{lemma}

Using the characters $\chi_1, \dots, \chi_n$ we define a map $\phi:\cP_{K, p} \to \cP_{K', p}$, where $\phi(\mfp_i)$ is the unique prime lying over $p$ at which $\psi(\chi_i)$ does not have value $1$. 

\begin{lemma}\label{lemma:phi norm preserving}
The map $\phi$ is norm-preserving and we have $\psi(\chi_i)(\phi(\mfp_i)) = \zeta_l = \chi_i(\mfp_i)$ for any $1 \leq i \leq n$.
\end{lemma}

\begin{proof}
By definition of $\chi_i$ we have
\[
L_p(\chi_i, T) = (1 - \zeta_l T^{f_i}) \prod_{j \neq i} (1 - T^{f_j}).
\]
The trivial character $1_K$ has a similar local $L$-series at $p$:
\[
L_p(1_K, T) = (1 - T^{f_i}) \prod_{j \neq i} (1 - T^{f_j}).
\]
Expanding both products, we see in particular that 
\[
[T^j]L_p(\chi_i, T) = [T^j]L_p(1_K, T) 
\]
for all $j < f_i$. Focusing on the coefficient of $T^{f_i}$ we find that
\begin{equation}\label{equation:K-side}
L_p(\chi_i, T) - L_p(1_K, T) = (1 - \zeta_l)T^{f_i} + \text{higher order terms}.
\end{equation}
We obtain a similar equality on the side of $K'$. Let $f'_{\phi(\mfp_j)}$ be the inertia degree of $\phi(\mfp_j)$. Then by Lemma~\ref{lemma:psichi only one prime}
\[
L_p(\psi(\chi_i), T) = \Big(1 - \psi(\chi_i)(\phi(\mfp_i)) T^{f'_{\phi(\mfp_i)}}\Big) \prod_{j \neq i} (1 - T^{f'_{\phi(\mfp_j)}}).
\]
Therefore,
\begin{equation}\label{equation:K'-side}
L_p(\psi(\chi_i), T) - L_p(1_{K'}, T) = \big(1 - \psi(\chi_i)(\phi(\mfp_i))\big)T^{f'_{\phi(\mfp_i)}} + \text{higher order terms}.
\end{equation}
Because $L_p(\chi_i, T) = L_p(\psi(\chi_i), T)$ and $L_p(1_K, T) = L_p(1_{K'}, T)$, we find by combining (\ref{equation:K-side}) and (\ref{equation:K'-side}) that $f_i = f'_{\phi(\mfp_i)}$ (thus $\phi$ is norm-preserving) and
\[
\psi(\chi_i)(\phi(\mfp_i)) = \zeta_l = \chi_i(\mfp_i). \qedhere
\]
\end{proof}

\begin{lemma}
The map $\phi: \cP_{K, p} \to \cP_{K', p}$ has the property that $\chi(\mfp) = \psi(\chi)(\phi(\mfp))$ for all $\chi \in \Hlab[l]$ and $\mfp \in \cP_{K, p}$. In particular $\phi$ is independent of the choice of the characters $\chi_i$. Moreover, $\phi$ is a bijection.
\end{lemma}

\begin{proof}
Let $1 \leq i \leq n$ and let $\chi \in \Hlab[l]$ be any character of order $l$. By definition of $\chi_i$, we have that $\chi\chi_i(\mfp) = \chi(\mfp)$ for any $\mfp \in \cP_{K, p}$ unequal to $\mfp_i$. This implies that the order of zero at $T = 1$ of $L_p(\chi, T)$ and $L_p(\chi \chi_i, T)$ differs by at most one, and this difference can be attributed entirely to the values of $\chi$ and $\chi_i$ at $\mfp_i$. Similarly by Lemma~\ref{lemma:psichi only one prime}, $\psi(\chi)(\mfq) = \psi(\chi \chi_i)(\mfq)$ for all $\mfq \mid p$ except a single prime $\phi(\mfp_i)$. Therefore the difference in the order of zero at $T = 1$ depends only on the values of $\psi(\chi)$ and $\psi(\chi_i)$.

Assume $\chi(\mfp_i) \neq 0$. We distinguish two cases for the value of $\chi(\mfp_i)$:
\begin{itemize}
\item Suppose $\chi(\mfp_i) = 1$. Then $\chi \chi_i(\mfp_i) = \zeta_l$, hence the order of zero at $T = 1$ of $L_p(\chi \chi_i, T)$ is one lower than the order of zero of $L_p(\chi, T)$ at $T = 1$. The same must therefore hold for $L_p(\psi(\chi \chi_i), T)$ and $L_p(\psi(\chi), T)$, hence $\psi(\chi)(\phi(\mfp_i)) = 1$.
\item Suppose $\chi$ is unramified at $\mfp_i$, and $\chi(\mfp_i) \neq 1$. Then $\chi(\mfp_i)$ is an $l^\text{th}$ root of unity. Let $k$ be such that $\chi(\mfp_i)^{-1} = \zeta_l^k = \chi_i(\mfp_i)^k$. Then $\chi\chi_i^k(\mfp_i) = 1$, while $\chi\chi_i^k(\mfp) = \chi(\mfp)$ for all $\mfp \neq \mfp_i$ lying over $p$. Thus the order of zero at $T = 1$ of $L_p(\chi \chi_i^k, T)$ is one higher than $L_p(\chi, T)$. The same therefore holds for $L_p(\psi(\chi \chi_i^k), T)$ and $L_p(\psi(\chi), T)$. Hence $\psi(\chi \chi_i^k)(\phi(\mfp_i)) = 1$, which combined with Lemma~\ref{lemma:phi norm preserving} shows that $\psi(\chi)(\phi(\mfp_i)) = \psi(\chi_i)(\phi(\mfp_i))^{-k} = \zeta_l^{-k} = \chi(\mfp_i)$.
\end{itemize}

This shows that $\chi(\mfp) = \psi(\chi)(\phi(\mfp))$ for all $\mfp$ lying over $p$ and all $\chi \in \Hlab[l]$, provided that $\chi(\mfp) \neq 0$.

We now show that $\phi: \cP_{K,p} \to \cP_{K', p}$ is a bijection. As $K$ and $K'$ have the same number of primes lying over $p$, it suffices to show $\phi$ is injective. Let $\mfp_i, \mfp_j \in \cP_{K, p}$ such that $\phi(\mfp_i) = \phi(\mfp_j)$. The character $\chi_i$ has value $\zeta_l$ at $\mfp_i$, hence $\psi(\chi_i)$ has value $\zeta_l$ at $\phi(\mfp_i)$ by the first part of this proof. However, $\chi_i$ has value $1$ at all primes in $\cP_{K, p}$ other than $\mfp_i$, hence 
\[
\chi_i(\mfp_j) = \psi(\chi_i)(\phi(\mfp_j)) = \psi(\chi_i)(\phi(\mfp_i)) \neq 1
\]
can only hold if $\mfp_i = \mfp_j$.

We end this proof by showing that if $\chi$ is ramified at $\mfp_i$, then $\psi(\chi)$ is ramified at $\phi(\mfp_i)$. Suppose $\chi$ is ramified at $\mfp_i$, i.e. $\chi(\mfp_i) = 0$. Then $\chi\chi_i(\mfp) = \chi(\mfp)$ for all $\mfp \in \cP_{K, p}$, hence $L_p(\chi, T) = L_p(\chi\chi_i, T)$. Thus we have $L_p(\psi(\chi), T) = L_p(\psi(\chi\chi_i), T)$. We already know that $\psi(\chi)(\phi(\mfp)) = \psi(\chi\chi_i)(\phi(\mfp))$ for any $\mfp \in \cP_{K, p}$ unequal to $\mfp_i$, hence
\[
\prod_{\substack{\mfp \in \cP_{K, p} \\ \mfp \neq \mfp_i}} L_{\phi(\mfp)}(\psi(\chi), T) = \prod_{\substack{\mfp \in \cP_{K, p} \\ \mfp \neq \mfp_i}} L_{\phi(\mfp)}(\psi(\chi\chi_i), T).
\]
As $\phi$ is a bijection $\cP_{K,p} \to \cP_{K', p}$ it follows from this equality combined with the equality $L_p(\psi(\chi), T) = L_p(\psi(\chi\chi_i), T)$ that $L_{\mfp_i}(\psi(\chi), T) = L_{\mfp_i}(\psi(\chi\chi_i), T)$, i.e.\ we have that $\psi(\chi)(\phi(\mfp_i)) = \psi(\chi\chi_i)(\phi(\mfp_i))$. As $\psi(\chi\chi_i)(\phi(\mfp_i)) = \zeta_l \cdot \psi(\chi)(\phi(\mfp_i))$, we conclude that  $\psi(\chi)(\phi(\mfp_i)) = 0$.
\end{proof}

This shows there is a norm-preserving bijection $\phi: \cP_K \to \cP_{K'}$ such that for all $\chi \in \Hlab[l]$ and $\mfp \in \cP_K$ we have
\[
\psi(\chi)(\phi(\mfp)) = \chi(\mfp).
\]
Hence $\Theta_3: \psi \mapsto \psi$ is a well-defined injective map.

This establishes the bijections $\Theta_1$, $\Theta_2$, and $\Theta_3$. We end this section with a corollary of the proven result.

\begin{corollary}
Let $\psi \in \IsoPx$. The associated injective norm-preserving map of primes $\phi: \cS \to \cP_{K'}$ can be extended uniquely to a bijection of primes $\phi: \cP_{K} \to \cP_{K'}$ for which
\[
\psi(\chi)(\phi(\mfp)) = \chi(\mfp)
\]
for all $\chi \in \Hlab[l]$ and $\mfp \in \cP_{K}$.
\end{corollary}

\begin{proof}
If $\phi$ extends to a bijection $\cP_K \to \cP_{K'}$, we know it extends uniquely by Lemma~\ref{lemma:prime-bijection-is-unique}. Following maps $\Theta_2$ and $\Theta_3$, we see that $\psi \in \IsoP$, i.e., there is a norm-preserving bijection $\phi': \cP_K \to \cP_{K'}$ such that for any $\chi \in \Hlab[l]$ and any $\mfp \in \cP_K$ we have 
\[
\chi(\mfp) = \psi(\chi)(\phi'(\mfp)).
\]
We show that $\phi'$ extends $\phi$ in a manner similar to Lemma~\ref{lemma:prime-bijection-is-unique}. Let $\mfp \in \cS$, then 
\[
\psi(\chi)(\phi(\mfp)) = \chi(\mfp) =  \psi(\chi)(\phi'(\mfp))
\]
for any $\chi \in \Hlab[l]$. If $\phi(\mfp) \neq \phi'(\mfp)$, then by Lemma~\ref{lemma:grunwald-wang in any order} there exists a character $\chi' \in \Hlabx[l]$ such that $\chi'(\phi(\mfp)) = 1$ and $\chi'(\phi'(\mfp)) \neq 1$. We obtain a contradiction by setting $\chi = \psi^{-1}(\chi')$. 
\end{proof}

\section{Maps $\Theta_4$ and $\Theta_5$}\label{sec:maps 4 and 5}
In this section we construct an injective maps between $\Isofield$ and $\IsoP$ and between $\IsoPx$ and $\Isofield$. The majority of the section is devoted to the map $\Theta_5$: it is constructed by first moving to the Galois closure, finding a suitable automorphism of this Galois closure, and then showing that this restricts to an isomorphism $K \DistTo K'$. 

Elements of $\Isofield$, $\IsoP$, and $\IsoPx$ come equipped with a bijection of primes, and they are uniquely determined by this bijection of primes (see Corollaries~\ref{cor:psi1-psi2} and~\ref{corollary:sigma is determined by bijection of primes}). To show that $\Theta_1 \circ \Theta_4$ and $\Theta_5$ are mutual inverses, we prove that both maps (after being defined) ``preserve the bijection of primes'', that is, we show that the bijection of primes associated to an element is the same as the bijection of primes of the image of this element under either $\Theta_1 \circ \Theta_4$ or $\Theta_5$.

\subsection*{The map $\Theta_4$} \label{sec:Isofield-to-IsoP}\leavevmode \\
Any field isomorphism $\sigma: K \DistTo K'$ has an associated bijection of primes $\cP_{K} \to \cP_{K'}$, also denoted by $\sigma$. Furthermore, $\sigma$ induces an isomorphism of Galois groups $\Gkabx \DistTo \Gkab$ given by conjugation with any $\tau: K^\textup{ab} \DistTo {\big(K'\big)}^\textup{ab}$ that is an extension of $\sigma$ (the map is independent of the choice of $\tau$). By dualizing we obtain a map\begin{equation}\label{equation:psi-sigma}
\psi_\sigma: \Hlab[l] \DistTo \Hlabx[l]. 
\end{equation}

Let $\chi \in \Hlab[l]$ be any character. Restricting $\tau$ to $K_\chi$ gives an isomorphism $K_{\chi} \DistTo K_{\psi_\sigma(\chi)}$. Hence $\chi$ is ramified at $\mfp \in \cP_{K}$ if and only if $\psi_\sigma(\chi)$ is ramified at $\sigma(\mfp)$. Let $\mfp \in \cP_{K}$ be a prime at which $\chi$ is unramified. Let $\Frob{\mfp} \in \Gal(K_{\chi}/K)$ be the Frobenius element at $\mfp$ and let $\Frob{\sigma(\mfp)} \in \Gal(K'_{\psi_\sigma(\chi)}/K')$ be the Frobenius element at $\sigma(\mfp)$. The aforementioned map $\Gkabx \DistTo \Gkab$ has a quotient $\Gal(K'_{\psi_\sigma(\chi)}/K') \DistTo \Gal(K_{\chi}/K)$ that maps $\Frob{\sigma(\mfp)}$ to $\Frob{\mfp}$. As a result, for any $\mfp \in \cP_{K}$ and $\chi \in \Hlab[l]$ we have $\psi_\sigma(\chi)(\sigma(\mfp)) = \chi(\mfp)$. Thus $\psi_\sigma \in \IsoP$, with corresponding bijection of primes $\sigma: \cP_K \to \cP_{K'}$. This proves the following corollary:

\begin{corollary}\label{corollary:psi has same bijection of primes as sigma}
For any field isomorphism $\sigma: K \DistTo K'$ there is a group isomorphism $\psi_\sigma: \Hlab[l] \DistTo \Hlabx[l]$ such that
\[
\psi_\sigma(\chi)(\sigma(\mfp)) = \chi(\mfp),
\]
for any $\chi \in \Hlab[l]$ and $\mfp \in \cP_K$. 
\end{corollary}

Define the map $\Theta_4$ by $\sigma \mapsto \psi_\sigma$. We prove that it is injective.

\begin{lemma}\label{lemma:sigma is determined by its bijection of primes}
Let $K/k$ be an extension of number fields, $\sigma \in \textup{Aut}(K/k)$, and $\mfp$ a prime that lies over a totally split prime in $k$. If $\sigma(\mfp) = \mfp$, then $\sigma$ is the identity. 
\end{lemma}

\begin{proof}
We first consider the case where $K/k$ is Galois. The action of $\Gal(K/k)$ on the primes of $K$ lying over $\mfp \cap k$ is transitive. Moreover, $\#\Gal(K/k) = [K:k]$ equals the number of primes lying over $\mfp \cap k$, thus an automorphism of $K$ is uniquely determined by the image of $\mfp$. 

For the general case, denote by $N$ the Galois closure (over $k$) of $K$. As $\mfp \cap k$ is totally split in $K$, it is totally split in all of the Galois conjugates of $K/k$, whose union is $N$. Thus $\mfp \cap k$ is totally split in $N$ as well. Let $\tau$ be any extension of $\sigma$ to $N$, and let $\mfP$ be any prime lying over $\mfp$. As $\sigma(\mfp) = \mfp$, $\tau$ permutes the primes lying over $\mfp$. This implies that there exists a $\tau' \in \Gal(N/K)$ such that $\tau \tau'(\mfP) = \mfP$, as $\Gal(N/K)$ acts transitively on the primes lying over $\mfp$. By the previous paragraph, $\tau \tau'$ is the identity. As $\tau' \in \Gal(N/K)$, we conclude that $\tau \in \Gal(N/K)$, hence $\sigma$ is the identity.
\end{proof}

\begin{corollary}\label{corollary:sigma is determined by bijection of primes}
Let $\sigma_1, \sigma_2 \in \Isofield$. If they induce the same bijection of primes, then they are equal.
\end{corollary}

\begin{proof}
Apply the previous lemma to $\sigma_1 \sigma_2^{-1}$.
\end{proof}

\begin{lemma}
The map $\Theta_4: \sigma \mapsto \psi_\sigma$ is injective.
\end{lemma}

\begin{proof}
Suppose $\sigma_1$ and $\sigma_2$ have the same image. By Corollary~\ref{corollary:psi has same bijection of primes as sigma} and Lemma~\ref{lemma:prime-bijection-is-unique}, $\sigma_1$ and $\sigma_2$ the induce the same bijection of primes, hence by Corollary~\ref{corollary:sigma is determined by bijection of primes} we have $\sigma_1 = \sigma_2$.
\end{proof}

\subsection*{The map $\Theta_5$} \leavevmode \\
\new{In this section we constuct a map $\Theta_5: \IsoPx \to \Isofield$ that is an inverse of the composite $\Theta_1 \circ \Theta_4: \Isofield \to \IsoPx$. For brevity we refer to Dirichlet density simply as ``density''. The main theorem we prove is Theorem~\ref{theoremA:pos_density}:}

\begin{theorem}\label{theorem:positive_density_result}
\new{Let $K$ and $K'$ be number fields of the same degree, and let $\cS$ be a subset of the primes of inertia degree $1$ of $K$. Suppose that for some finite extension $L/K$, $\cS$ contains $\emph{\text{Spl}}(L/K)$ except for a density zero set. Furthermore, suppose there exists an isomorphism $\psi: \Hlab[l] \DistTo \Hlabx[l]$ with an injective norm-preserving map $\phi: \cS \to \cP_{K'}$ such that
\[
\chi(\mfp) = \psi(\chi)(\phi(\mfp))
\]
Then $K \simeq K'$ and there is a unique $\sigma_\psi: K \DistTo K'$ such that the bijection of primes induced by $\sigma_\psi$ equals $\phi$ on $\cS$ except for finitely many exceptions. }
\end{theorem}

\new{We begin by deducing that the desired result follows from this theorem.}

\begin{corollary}
\new{The map $\Theta_5: \IsoPx \to \Isofield$ given by $\psi \mapsto \sigma_\psi$ is a (two-sided) inverse to $\Theta_1 \circ \Theta_4$. }
\end{corollary}

\begin{proof}
\new{Let $\psi \in \IsoPx$. It has an associated injective norm-preserving map of primes $\phi: \cS \to \cP_{K'}$, where $\cS \subseteq \cP_K$ is of density one. Let $\cS'$ be the image of $\phi$. Lemma~\ref{lemma:density preserved by norm-preserving bijection} below guarantees that the density of $\cS'$ is one as well. As $\cS$ has density one, there is a rational prime $p$ unramified in $K/\Qz$ such that all primes lying over $p$ are contained in $\cS$. The degree of $K$ equals the sum of all inertia degrees of the primes lying over $p$, and as $\phi$ is norm-preserving and injective it follows that $[K':\Qz] \geq [K:\Qz]$. By repeating this argument for $\cS'$ and $\phi^{-1}: \cS' \to \cS$, we obtain $[K':\Qz] = [K:\Qz]$.}

\new{The theorem now provides an isomorphism $\sigma_\psi$ because $\phi$ is defined on all but a zero density subset of the primes of inertia degree $1$.}

\new{Let $\psi \in \IsoPx$ and consider $\psi_{\sigma_{\psi}} = (\Theta_1 \circ \Theta_4 \circ \Theta_5)(\psi)$. By Theorem~\ref{theorem:positive_density_result} and Corollary~\ref{corollary:psi has same bijection of primes as sigma} we find that the maps of primes associated to $\psi_{\sigma_{\psi}}$ and $\psi$ agree everywhere except on a density zero set of primes. It follows from Corollary~\ref{cor:psi1-psi2} that $\psi_{\sigma_{\psi}} = \psi$. The injectivity of $\Theta_1 \circ \Theta_4$ guarantees that $\Theta_5 \circ \Theta_1 \circ \Theta_4$ is the identity map as well.}
\end{proof} 

\begin{lemma}\label{lemma:density preserved by norm-preserving bijection}
Let $K$ and $K'$ be number fields. Let $\cS \subseteq \cP_{K}$. Suppose there is an injective norm-preserving map $\phi: \cS \to \cP_{K'}$, and assume that $\cS$ has a Dirichlet density $\delta(\cS)$ in the primes of $K$. Then the density of the image of $\phi$ in $\cP_{K'}$ exists and equals $\delta(\cS)$. 
\end{lemma}

\begin{proof}
Denote the image of $\phi$ by $\cS'$. Let $s \in \Cz$ with $\text{Re}(s) > 1$. 
By \cite[Ch.\ 7, \S 2, Cor.\ 1]{narkiewicz}, we have for $s \to 1$ that
\[
\frac{\sum_{\mfp \in \cP_K} N\mfp^{-s}}{\log(1/(s-1))} \to 1.
\]
The same holds for $\sum_{\mfp \in \cP_{K'}} N\mfq^{-s} $. 

As $\cS$ has Dirichlet density $\delta(\cS)$, we have
\[
\lim_{s \to 1} \frac{\sum_{\mfp \in \cS} N\mfp^{-s}}{\sum_{\mfp \in \cP_K} N\mfp^{-s} } = \delta(\cS). 
\] Moreover, as $\phi$ is a norm-preserving bijection $\cS \to \cS'$, 
\[
\sum_{\mfp \in \cS} N\mfp^{-s} = \sum_{\mfp \in \cS} N\phi(\mfp)^{-s} = \sum_{\mfq \in \cS'} N\mfq^{-s}.
\]
It follows that
\[
\lim_{s \to  1} \frac{\sum_{\mfq \in \cS'} N\mfq^{-s}}{\sum_{\mfp \in \cP_{K'}} N\mfq^{-s}} = \lim_{s \to  1} \frac{\sum_{\mfq \in \cS'} N\mfq^{-s}}{\log(1/(s-1))} =  \lim_{s \to 1} \frac{\sum_{\mfp \in \cS} N\mfp^{-s}}{\sum_{\mfp \in \cP_K} N\mfp^{-s} } = \delta(\cS). \qedhere
\]
\end{proof}

\new{The idea of the proof of Theorem~\ref{theorem:positive_density_result} is as follows: let $N$ be a Galois extension of $\Qz$ containing $K'$ and $L$ (hence also $K$). We show that there is a $\sigma \in \Gal(N/\Qz)$ that ``agrees'' with $\phi$ on many primes (this is made more precise in Definition~\ref{def:mirror}), and that this $\sigma$ restricts to an isomorphism $K \DistTo K'$. This approach is motivated by the following lemma, which makes a descent from $N$ to $K$ possible based only on a condition on the primes:}

\begin{lemma}\label{lem:restriction}
\new{Let $N/\Qz$ be a Galois extension, and $K$ and $K'$ subfields of $N$ of the same degree. Let $\sigma \in \Gal(N/\Qz)$, $p$ be any rational prime that is totally split in $N$, and fix a prime $\mfp \mid p$ of $K$. Suppose that $\mfq := \sigma(\mfP) \cap K'$ is independent of the choice of the prime $\mfP \mid \mfp$ of $N$. Then $\sigma$ restricts to an isomorphism $K \DistTo K'$.}
\end{lemma}

\begin{proof}
\new{By assumption $\sigma$ maps any prime of $N$ lying over $\mfp$ to a prime of lying over $\mfq$. Because $p$ is totally split in $N/\Qz$, both $\mfp$ and $\mfq$ split into $[N:K] = [N:K']$ primes in $N$. Therefore $\sigma$ bijectively maps the primes lying over $\mfp$ to the primes lying over $\mfq$.} 

\new{Let $\tau \in \Gal(N/K')$ and consider the action of $\sigma^{-1} \tau \sigma$ on the primes of $N$ lying over $p$. If $\mfP$ is any prime lying over $\mfp$, then $(\tau \sigma)(\mfP)$ lies over $\mfq$, hence by the previous paragraph $(\sigma^{-1}\tau\sigma)(\mfP)$ lies over $\mfp$. Therefore $\sigma^{-1} \tau \sigma$ permutes the primes lying over $\mfp$. As $p$ is totally split in $N$, an element in $\Gal(N/\Qz)$ is uniquely determined by the image of $\mfP$,. Because $\Gal(N/K)$ acts transitively on the primes lying over $\mfp$, it follows that $\sigma^{-1}\tau \sigma \in \Gal(N/K)$. Hence we find
\[
\Gal(N/\sigma^{-1} K') = \sigma^{-1} \Gal(N/K') \sigma \subseteq \Gal(N/K),
\]
and thus $\sigma^{-1} K' = K$ as the sets have equal cardinality.}
\end{proof}

For the remainder of this section, assume that the conditions of Theorem~\ref{theorem:positive_density_result} hold.

\begin{definition}\label{def:mirror}
Let $\sigma \in \Gal(N/\Qz)$ and $\mfP$ a prime of $N$ lying over $\mfp$, a prime of $K$ \new{contained in $\cS$}. We say that $\sigma$ \emph{follows} $\phi$ \emph{at} $\mfP$ if $\sigma(\mfP)$ lies over $\phi(\mfp)$. 
\end{definition}

\begin{lemma}
There exists a $\sigma \in \Gal(N/\Qz)$ such that the set
\[
\cP_\sigma := \{\mfP \in \cP_N: f_{\mfP} = 1\text{, } \mfP \cap K \in \cS \text{ and } \sigma \text{ follows } \phi \text{ at } \mfP\}
\]
has positive density in the primes of $N$.
\end{lemma}

\begin{proof}
\new{Note that for any $\mfP \in \cP_N$ with inertia degree $1$ we have that $\mfP \cap K$ splits completely in $N/K$, hence $\mfP \cap K \in \text{Spl}(L/K)$. As $\cS$ contains $\text{Spl}(L/K)$ except for a set of density zero, the set
\[
\cS_N = \{\mfP \in \cP_N: f_{\mfP} = 1\text{, } \mfP \cap K \in \cS\}
\]
has density one in the primes of $N$.}

As $N/\Qz$ is a Galois extension, $\Gal(N/\Qz)$ acts transitively on the primes lying over any rational prime $p$. In particular, for any prime $\mfP \mid \mfp$ there exists a $\sigma$ such that $\sigma(\mfP) \mid \phi(\mfp)$. Hence
\[
\bigcup_{\sigma \in \Gal(N/\Qz)} \cP_\sigma
\]
is the entire set \new{$\cS_N$} and \new{therefore} a set of density one. Because $\Gal(N/\Qz)$ is finite, there must exist a $\sigma$ for which the set $\cP_\sigma$ has positive density, in fact the density will be at least $1/[N:\Qz]$.
\end{proof}

For the remainder of this section, we fix $\sigma$ to be the element of $\Gal(N/\Qz)$ found in the previous lemma. In order to relate this automorphism of $N$ to the character group $\Hlab[l]$, we \new{make use of} a map $\Hlab[l] \to \Hnab[l]$ that is ``almost injective'', which allows us to study the action of $\sigma$ on the characters of $G_K$.

\begin{definition}\label{def:overline-chi}
Define $\iota_{K, N}: \Hkab \to \Hnab$ as the dual of the map 
\begin{align*}
\Gnab &\to \Gkab\\
\sigma &\mapsto \restr{\sigma}{K^\textup{ab}}.
\end{align*}
For brevity we denote $\iota_K := \iota_{K, N}$ and $\iota_{K'} := \iota_{K', N}$. 
\end{definition}

\begin{remark}\label{remark:overline-chi}
The extension $N_{\iota_K(\chi)}/N$ associated to $\iota_K(\chi)$ is the composite of fields $N$ and $K_\chi$, hence the (finite) kernel of $\iota_K$ consists precisely of the characters $\chi$ such that $K_{\chi} \subseteq N$. \new{If $\chi$ is a character of order $l$}, $N_{\iota_K(\chi)}/N$ is an extension of degree either $1$ or $l$. If $\mfP$ is a prime of $N$ such that $\chi$ is unramified at $\mfp := \mfP \cap K$, then $\iota_K(\chi)(\mfP) = \chi(\mfp)^{[\mathbb{F}_{\mfP}:\mathbb{F}_\mfp]}$. In particular, if $\mfp$ splits completely in $N/K$, then $\iota_K(\chi)(\mfP) = \chi(\mfp)$.
\end{remark}

\begin{lemma}\label{lem:finite-image}
The map
\begin{align*}
\alpha: \Hlab[l] &\to \Hnab[l] \\
\chi &\mapsto \psi_\sigma(\iota_K(\chi))/ \iota_{K'}(\psi(\chi)),
\end{align*}
where $\psi_\sigma = \Theta_4(\sigma)$ is defined as in \textup{(\ref{equation:psi-sigma})}, has finite image. As a result, its kernel has finite index.
\end{lemma}

\begin{proof}
Denote the density of $\cP_\sigma$ by $\delta$. For every prime in $\cP_\sigma$ we have that $\sigma$ follows $\phi$, i.e. if $\mfP \in \cP_\sigma$ and $\mfP \mid \mfp$, then $\sigma(\mfP) \mid \phi(\mfp)$. Let $\chi \in \Hkab[l]$ be any character, and let $\mfP \in \cP_\sigma$ such that $\chi$ is unramified at $\mfp := \mfP \cap K$. 
As $\mfP$ and therefore $\sigma(\mfP)$ have inertia degree $1$, $\mfp$ and $\phi(\mfp)$ split completely in $N/K$. By Remark~\ref{remark:overline-chi} we have
\[
\iota_{K'}(\psi(\chi))(\sigma(\mfP)) = \psi(\chi)(\phi(\mfp)) = \chi(\mfp) = \iota_K(\chi)(\mfP) = \psi_\sigma(\iota_K(\chi))(\sigma(\mfP)).
\]
thus $\alpha(\chi)(\sigma(\mfP)) = 1$, i.e. $\sigma(\mfP)$ splits in the extension associated to the character $\alpha(\chi)$. Note that as $\cP_{\sigma}$ has density $\delta$, so does $\sigma\big(\cP_{\sigma}\big)$, e.g. by Lemma~\ref{lemma:density preserved by norm-preserving bijection}. 

We now argue by contradiction: suppose the image contains infinitely many characters. Choose infinitely many different $\theta_1, \theta_2, \dots \in \Hnab[l]$ in the image of $\alpha$. Denote by $N_{\theta_i}$ the extension associated to $\theta_i$. The composite of all $N_{\theta_i}$ is an infinite extension of $N$, as only finitely many (different) characters can factor through a finite Galois group. Therefore there is an $n \in \Nz$ such that the composite $M = N_{\theta_1} N_{\theta_2} \dots N_{\theta_n}$ (which is Galois over $N$) is of degree larger than $1/\delta$ over $N$. Each of these characters ramifies at finitely many primes, hence for any $1 \leq i \leq n$ the character $\theta_i$ has value $1$ on all but finitely many primes of $\sigma\big(\cP_\sigma\big)$. This implies that there is a density $\delta$ subset of the primes of $N$ such that every prime in this set splits completely in each of the extensions $N_{\theta_i}/N$. Hence all these primes also split completely in $M/N$, and by the Chebotarev density theorem we have $[M:N] \leq 1/\delta$, which is a contradiction.
\end{proof}

The finite index subgroups of $\Hkab[l]$ all contain many characters of a specific form, as made precise by the following definition and lemma.

\begin{definition}\label{definition:X_p}
Let $p$ be a rational prime, $\mfp \in \cP_{K, p}$, and define $X_\mfp \subseteq \Hlab[l]$ as the set of characters $\chi$ of order $l$ with the following properties:
\begin{enumerate}
\item $\chi(\mfp) = \zeta_l \, (= \exp(2\pi i/l))$; and
\item $\chi(\tilde{\mfp}) = 1$ for all $\tilde{\mfp} \in \cP_{K, p}$ unequal to $\mfp$.
\end{enumerate}
By Lemma~\ref{lemma:grunwald-wang in any order}, this set is non-empty.
\end{definition}

\begin{lemma}\label{finite-index-has-chip}
Let $U$ be a finite index subgroup of $\Hlab[l]$. Then for all but finitely many primes $\mfp$ of $K$, there is a character $\chi_\mfp \in X_{\mfp}$ contained in $U$.
\end{lemma}

\begin{proof}
We give a proof by contradiction. Enumerate the primes for which the condition does not hold (there are infinitely many by assumption) by $\mfp_1, \mfp_2, \dots$ and define $p_i = \mfp_i \cap \Zz$. As only finitely many primes of $K$ lie over a certain $p_i$, we may assume (by removing some of the $\mfp_i$ if necessary) that all $p_i$ are different. 

 Using Lemma~\ref{lemma:grunwald-wang in any order}, we create for any $i \in \Nz$ a character $\chi_i \in \Hlab[l]$ such that
\begin{itemize}
\item $\chi_i$ is unramified at all primes lying over $p_1, p_2, \dots, p_i$; and
\item $\chi_i$ has value $1$ at all primes of $K$ lying over $p_1, \dots p_i$, except $\mfp_i$, where it has value $\zeta_l$.
\end{itemize}
We have $\chi_i \in X_{\mfp_i}$ and by assumption $\chi_i \notin U$. We show that $\chi_i$ and $\chi_j$ do not lie in the same coset of $\Hlab[l]/U$. Assume without loss of generality that $i < j$. By construction $\chi_j$ has value $1$ at all primes of $K$ lying over $p_i$. Moreover $\chi_i \chi_j^{-1}(\mfp_i) = \zeta_l$, thus we have $\chi_i \chi_j^{-1} \in X_{\mfp}$, hence by assumption $\chi_i\chi_j^{-1} \notin U$. It follows that $\Hlab[l]/U$ has infinitely many cosets; a contradiction.
\end{proof}

\begin{lemma}
The automorphism $\sigma \in \Gal(N/\Qz)$ restricts to an isomorphism $K \DistTo K'$.
\end{lemma}

\begin{proof}
\new{We consider two criteria for rational primes $p$:}
\begin{enumerate}
\item the prime $p$ lies in $\text{Spl}(N/\Qz)$ and every prime of $N$ lying over $p$ is contained in $\cS_N$;
\item for any $\mfp \in \cP_{K, p}$ there exists a character $\chi_\mfp \in X_\mfp$ such that $\chi_{\mfp} \in \ker(\alpha)$.
\end{enumerate}
\new{As $\cS_N$ has density $1$ in the primes of $N$, there exist infinitely many rational primes $p$ that meet the first criterium. The kernel of the map $\alpha$ of Lemma~\ref{lem:finite-image} has finite index, hence we can apply the previous lemma, from which it follows that all but finitely many rational primes $p$ meet the second criterium. Hence there exists a rational prime $p$ that meets both. Fix this prime for the remainder of the proof.}

\new{Because $p$ splits in $N/\Qz$, it also splits in both $K/\Qz$ and $K'/\Qz$. As all the primes of $N$ lying over $p$ are contained in $\cS_N$, all primes of $K$ lying over $p$ are contained in $\cS$. As a result, $\phi$ bijects the primes of $K$ and the primes of $K'$ lying over $p$.}

\new{Let $\mfp$ be any prime of $K$ lying over $p$. By Lemma~\ref{lem:restriction} it suffices to prove that $\sigma(\mfP) \cap K'$ is independent of the choice of $\mfP \mid \mfp$. We prove this by showing that $\sigma(\mfP) \cap K' = \phi(\mfp)$. The character $\chi_{\mfp}$ has value $1$ on all primes lying over $p$ except $\mfp$, and because all primes of $K$ lying over $p$ are contained in $\cS$ it follows that $\psi(\chi_\mfp)$ has value $1$ on all primes of $K'$ lying over $p$ except $\phi(\mfp)$. }
 
Because $\chi_\mfp \in \ker(\alpha)$ we have 
\[
\iota_{K'}(\psi(\chi_{\mfp})) = \psi_\sigma(\iota_K(\chi_\mfp)),
\]
which combined with Remark~\ref{remark:overline-chi} yields the following chain of equalities:
\[
\psi(\chi_\mfp)(\sigma(\mfP) \cap K') = \iota_{K'}(\psi(\chi_\mfp))(\sigma(\mfP)) = \psi_\sigma(\iota_K(\chi_\mfp))(\sigma(\mfP)) = \iota_K(\chi_\mfp)(\mfP) = \chi_\mfp(\mfp).
\]
The value of $\chi_\mfp(\mfp)$ is unequal to $1$ by construction, hence $\psi(\chi_\mfp)(\sigma(\mfP) \cap K) \neq 1$. However, $\psi(\chi_\mfp)$ has value $1$ on all primes lying over $p$ except $\phi(\mfp)$, hence we conclude that $\sigma(\mfP) \cap K = \phi(\mfp)$. 
\end{proof}

\new{We assert that $\restr{\sigma}{K}$ induces the same map of primes $\cS \to \cP_{K'}$ as $\phi$ except for finitely many exceptions. The composite map}
\begin{alignat*}{2}
\Hlab[l] &\to \hspace{18pt}\Hlab[l]\hspace{18pt} &&\to \Hnab[l] \\
\chi &\mapsto \psi_{\sigma \mid_K}(\chi)/\psi(\chi) &&\mapsto \iota_{K'}\big(\psi_{\sigma \mid_K}(\chi)/\psi(\chi)\big) = \psi_{\sigma}(\iota_K(\chi))/\iota_{K'}(\psi(\chi))
\end{alignat*}
\new{is the map $\alpha$ described in Lemma~\ref{lem:finite-image}, hence has finite image. As the second map has finite kernel by Remark~\ref{remark:overline-chi}, the first map must have finite image as well. This implies that the kernel $\{\chi: \psi_{\sigma \mid_K}(\chi) = \psi(\chi)\}$ of the map $\chi \mapsto \psi_{\sigma \mid_K}(\chi)/\psi(\chi)$ is of finite index in $\Hlab[l]$. An application of Lemma~\ref{finite-index-has-chip} guarantees that for all but finitely many primes $\mfp \in \cS$ we have a character $\chi_{\mfp} \in X_{\mfp}$ such that
\[
\psi_{\sigma \mid_K}(\chi_{\mfp}) = \psi(\chi_{\mfp}).
\]
Because $\psi_{\sigma \mid_K}(\chi_{\mfp})(\restr{\sigma}{K}(\tilde{\mfp})) = \chi(\tilde{\mfp})$ for any prime $\tilde{\mfp} \in \cP_K$, the character $\psi_{\sigma \mid_K}(\chi_{\mfp})$ has value $1$ on all primes of $K'$ lying over $\restr{\sigma}{K}(\mfp) \cap \Zz$, except $\restr{\sigma}{K}(\mfp)$, where it has value $\zeta_l$. However, 
\[
\psi_{\sigma \mid_K}(\chi_{\mfp})(\phi(\mfp)) = \psi(\chi_{\mfp})(\phi(\mfp)) = \chi(\mfp) = \zeta_l,
\]
hence $\phi(\mfp) = \restr{\sigma}{K}(\mfp)$. \qed}

To conclude the proof of Theorem~\ref{theorem:positive_density_result} we show that there is only one isomorphism $K \DistTo K'$ with this property. Suppose we have two such isomorphisms, say $\sigma_1$ and $\sigma_2$. They induce the same map of primes on all but a density zero set of $\text{Spl}(L/K)$. In particular there is a prime $p$ that is totally split in $K$ and a prime $\mfp \mid p$ such that $\sigma_1(\mfp) = \sigma_2(\mfp)$. An application of Lemma~\ref{lemma:sigma is determined by its bijection of primes} shows that $\sigma_1 = \sigma_2$.

\begin{remark}\label{remark:counterexample}
\new{There exist isomorphisms $\psi: \Hlab[l] \DistTo \Hlabx[l]$ which have an associated map of primes on a positive density set that cannot be extended to the full set $\cP_K$. For example, let $K = K' = \Qz$, $l = 2$, and let $p$ be a rational prime congruent to $1 \text{ mod } 4$. The characters of order $2$ are in bijection with the elements of $\Qz^\times / \big( \Qz^\times \big)^2$. For $d \in \Qz^\times / \big( \Qz^\times \big)^2$, let $\chi_{\sqrt{d}}$ be the character associated to $\Qz(\sqrt{d})/\Qz$. Define the map}
\begin{align*}
\psi: \widecheck{G}_{\Qz}[2] &\to \widecheck{G}_{\Qz}[2] \\
\chi_{\sqrt{d}} &\mapsto  
\begin{cases}
\chi_{\sqrt{d}} & \text{ if } \left(\frac{d \cdot |d|_p}{p} \right) = 1; \\
\chi_{\sqrt{-d}}& \text{ if } \left(\frac{d \cdot |d|_p}{p} \right) = -1.\\
\end{cases}
\end{align*}
\new{This map is an isomorphism ($\psi \circ \psi = \text{id}_{\widecheck{G}_{\Qz}[2]}$) that abides $\psi(\chi)(q) = \chi(q)$ for all primes $q$ congruent to $1 \text{ mod } 4$. For primes congruent to $3 \text{ mod } 4$ this equality does not hold, hence the prime bijection cannot be extended. A similar construction can be made for characters of higher order.}
\end{remark}

\section*{Acknowledgements}
We would like to thank Gabriele Dalla Torre for sharing his proof of Theorem~\ref{theorem:main theorem} for the case $l = 2$, and Gunther Cornelissen for many helpful discussions and remarks.

\end{document}